\documentclass[12pt]{amsart}%
\usepackage{amsmath}
\usepackage{amsfonts}
\usepackage{amssymb}
\usepackage{graphicx}
\usepackage{comment}%
\providecommand{\U}[1]{\protect\rule{.1in}{.1in}}
\newtheorem{theorem}{Theorem}
\newtheorem{acknowledgement}[theorem]{Acknowledgement}

\newtheorem{example}[theorem]{Example}

\newtheorem{lemma}[theorem]{Lemma}

\allowdisplaybreaks[1]

\theoremstyle{definition}
\numberwithin{equation}{section}

\newcommand{\resumename}{R\'esum\'e}

\begin{document}
\date{\today}
\title[Bandlimited Wavelets]{Bandlimited Wavelets on the Heisenberg Group}
\author[V. Oussa]{Vignon S. Oussa}
\address{Dept.\ of Mathematics\\
Bridgewater State University\\
Bridgewater, MA 02324 U.S.A.\\}
\email{Vignon.Oussa@bridgew.edu}
\maketitle

\begin{abstract}
Let $\mathbb{H}$ be the three-dimensional Heisenberg group. We introduce a
structure on the Heisenberg group which consists of the biregular
representation of $\mathbb{H\times H}$ restricted to some discrete subset of
$\mathbb{H\times H}$ and a free group of automorphisms $H$ singly generated
and acting semi-simply on $\mathbb{H}$. Using well-known theorems borrowed
from Gabor theory, we are able to construct simple and computable bandlimited
discrete wavelets on the Heisenberg group. Moreover, we provide some necessary
and sufficient conditions for the existence of these wavelets.

\end{abstract}




\section{Introduction}

A wavelet frame is a system generated by the action of translation and
dilation of a single function. More precisely, if $\psi\in L^{2}\left(
\mathbb{R}\right)  $ and $a,b$ are some fixed positive numbers and if
\begin{equation}
\mathcal{W}\left(  \psi,a,b\right)  =\left\{  a^{n/2}\psi\left(
a^{n}x-bk\right)  :k,n\in\mathbb{Z}\right\}  \label{affine}%
\end{equation}
is a frame (orthonormal basis) for $L^{2}\left(  \mathbb{R}\right)  $ then we
call $\mathcal{W}\left(  \psi,a,b\right)  $ a wavelet frame (orthonormal
basis). This system provides expansions for functions in $L^{2}\left(
\mathbb{R}\right)  .$ For example, in the case where $\mathcal{W}\left(
\psi,a,b\right)  $ is a Parseval frame or an orthonormal basis, then it is
known that for any function $\phi\in L^{2}\left(
\mathbb{R}
\right)  ,$
\[
\phi=\sum_{k,n\in%
\mathbb{Z}
}\left\langle \phi,a^{n/2}\psi\left(  a^{n}\cdot-bk\right)  \right\rangle
a^{n/2}\psi\left(  a^{n}\cdot-bk\right)  .
\]
Wavelet theory does extend to commutative groups of higher dimensions and even
to some non-commutative locally compact groups. In fact, the existence of
wavelets has been proved on the Heisenberg group and on other type of
stratified nilpotent Lie groups (see \cite{Mayeli} \cite{oussav},
\cite{Jawerth}, \cite{Hept} and \cite{Yang}). In \cite{Hept}, the authors
developed a theory of multiresolution analysis, by applying a concept of
acceptable dilations on the Heisenberg group which was used in the commutative
case by Gr\"{o}chenig and Madych in \cite{Grog}. As a result, they were able
to provide a description of Haar wavelets on the Heisenberg group. In
\cite{Mayeli}, Mayeli used multiresolution-like analysis and some sampling
theorems for the Heisenberg group obtained by F\"{u}hr in \cite{Fuhr cont} to
prove the existence of Shannon-type wavelets generated by discrete
translations and dilations on the Hilbert space of all square-integrable
functions on the Heisenberg group. However, there are several complications
which make this approach of construction of wavelets on the Heisenberg group
difficult. Let us be more precise. Since the Fourier and Plancherel transforms
are the central tools employed in the construction of wavelets, the
non-commutative nature of the Heisenberg group represents a major obstruction.
In fact, the left regular representation of the Heisenberg group is decomposed
via the Plancherel transform into a direct integral of Schr\"{o}dinger
representations, each occurring with infinite multiplicities. It turns out
that the structure which consists of the left regular representation
restricted to some discrete subgroup together with the usual expansive
dilation of the Heisenberg group is not a natural structure for the
construction of Shannon-type wavelets on the Heisenberg group. Therefore, in
order to construct simple and computable wavelets on the Heisenberg group,
there is a need to seek a different approach.

In the present work, we introduce a new construction of discrete bandlimited
wavelets on the Heisenberg group generated by a single function. The
techniques developed in this work are quite different from the ones known for
commutative groups and the ones employed by Mayeli in \cite{Mayeli}. The
structure considered in the present work consists of a pair $\left(
\tau,D\right)  $ where $D$ is a unitary representation induced by an
automorphism of the Heisenberg group, and $\tau$ is the biregular
representation of the Heisenberg group. The advantage of this new approach is
that, using available facts borrowed from Gabor theory, we are able to obtain
simple and computable discrete bandlimited wavelets (Shannon-like) on the
Heisenberg group.

Let us summarize our main result. Let $\mathbb{H}_{0}$ be the $3$-dimensional
Heisenberg group with Lie algebra $\mathfrak{h}$ spanned by $Z,Y,X$ such that
the only non-trivial Lie brackets are $\left[  X,Y\right]  =Z.$ We may think
of the Heisenberg group as being isomorphic to the non-commutative group
$\left(
\mathbb{R}
^{3},\ast\right)  $ with group law defined by
\[
\left(  x,y,z\right)  \ast\left(  w,v,u\right)  =\left(
w+x,v+y,u+z+vx\right)  .
\]
We introduce a convenient faithful finite-dimensional representation of the
Heisenberg group. Let us define an injective homomorphism $\rho:\mathbb{H}%
_{0}\rightarrow GL\left(  4,%
\mathbb{R}
\right)  $ such that
\begin{align*}
\rho\left(  \exp\left(  zZ\right)  \exp\left(  yY\right)  \exp\left(
xX\right)  \right)   &  =\left[
\begin{array}
[c]{cccc}%
1 & 0 & 0 & z\\
0 & 1 & 0 & 0\\
0 & 0 & 1 & 0\\
0 & 0 & 0 & 1
\end{array}
\right]  \left[
\begin{array}
[c]{cccc}%
1 & 0 & -y & 0\\
0 & 1 & 0 & y\\
0 & 0 & 1 & 0\\
0 & 0 & 0 & 1
\end{array}
\right]  \left[
\begin{array}
[c]{cccc}%
1 & x & 0 & 0\\
0 & 1 & 0 & 0\\
0 & 0 & 1 & 0\\
0 & 0 & 0 & 1
\end{array}
\right]  \\
&  =\left[
\begin{array}
[c]{cccc}%
1 & x & -y & z\\
0 & 1 & 0 & y\\
0 & 0 & 1 & 0\\
0 & 0 & 0 & 1
\end{array}
\right]  .
\end{align*}
Since $\rho$ has trivial kernel, then $\mathbb{H}=\rho\left(  \mathbb{H}%
_{0}\right)  $ is isomorphic to $\mathbb{H}_{0}.$ Next, we endow the group
$\mathbb{H}$ with its canonical Haar measure which is just like the Lebesgue
measure on $%
\mathbb{R}
^{3}$. Also, it is easy to see that the center of the Heisenberg group is
\[
Z\left(  \mathbb{H}\right)  =\left\{  \left[
\begin{array}
[c]{cccc}%
1 & 0 & 0 & z\\
0 & 1 & 0 & 0\\
0 & 0 & 1 & 0\\
0 & 0 & 0 & 1
\end{array}
\right]  :z\in%
\mathbb{R}
\right\}  .
\]
Finally, we define a discrete subgroup $\Gamma$ of the Heisenberg group as
follows:
\[
\Gamma=\left\{  \left[
\begin{array}
[c]{cccc}%
1 & k_{3} & -k_{2} & k_{1}\\
0 & 1 & 0 & k_{2}\\
0 & 0 & 1 & 0\\
0 & 0 & 0 & 1
\end{array}
\right]  :k_{1},k_{2},k_{3}\in%
\mathbb{Z}
\right\}  .
\]
Put%
\[
\Lambda=\left\{  \left(  \left[
\begin{array}
[c]{cccc}%
1 & k_{3} & -k_{2} & k_{1}\\
0 & 1 & 0 & k_{2}\\
0 & 0 & 1 & 0\\
0 & 0 & 0 & 1
\end{array}
\right]  ,\left[
\begin{array}
[c]{cccc}%
1 & m_{3} & -m_{2} & 0\\
0 & 1 & 0 & m_{2}\\
0 & 0 & 1 & 0\\
0 & 0 & 0 & 1
\end{array}
\right]  \right)  :\left[
\begin{array}
[c]{c}%
k_{1}\\
k_{2}\\
k_{3}\\
m_{2}\\
m_{3}%
\end{array}
\right]  \in%
\mathbb{Z}
^{5}\right\}  \subset\Gamma\times\Gamma
\]
and let
\[
A=\left[
\begin{array}
[c]{cccc}%
2 & 0 & 0 & 0\\
0 & b & 0 & 0\\
0 & 0 & a & 0\\
0 & 0 & 0 & 1
\end{array}
\right]  \in GL\left(  4,%
\mathbb{R}
\right)
\]
such that $ab=2.$ Then the map
\[
M\mapsto AMA^{-1}%
\]
defines an outer automorphism on the Heisenberg group. Now, let $\tau
:\mathbb{H\times H}\rightarrow\mathcal{U}\left(  L^{2}\left(  \mathbb{H}%
\right)  \right)  $ such that $\tau\left(  u,v\right)  f=L\left(  u\right)
R\left(  v\right)  f$ where
\[
L\left(  u\right)  f\left(  x\right)  =f\left(  u^{-1}x\right)  \text{ and
}R\left(  v\right)  f\left(  x\right)  =f\left(  xv\right)  .
\]
Clearly $\tau$ is the biregular representation of the Heisenberg group. Next,
define a representation $D$ of the group generated by $A$ such that
$D:\left\langle A\right\rangle \rightarrow\mathcal{U}\left(  L^{2}\left(
\mathbb{H}\right)  \right)  $ and
\[
D_{A^{m}}f\left(  n\right)  =\left\vert \delta\left(  A\right)  \right\vert
^{-m/2}f\left(  A^{-m}nA^{m}\right)
\]
where
\[
d\left(  A^{m}nA^{-m}\right)  =\left\vert \delta\left(  A\right)  \right\vert
^{m}dn
\]
and $dn$ is the canonical Haar measure on the Heisenberg group.

The main objective of the present paper is to prove the existence, and to find
characteristics of functions $f$ in $L^{2}\left(  \mathbb{H}\right)  $ such
that
\[
\left\{  D_{A^{m}}\tau\left(  \gamma,\eta\right)  f:m\in%
\mathbb{Z}
,\left(  \gamma,\eta\right)  \in\Lambda\right\}
\]
is a Parseval frame in $L^{2}\left(  \mathbb{H}\right)  .$ That is, given
$h\in L^{2}\left(  \mathbb{H}\right)  ,$%
\[
\sum_{m\in%
\mathbb{Z}
}\sum_{\left(  \gamma,\eta\right)  \in\Lambda}\left\vert \left\langle
h,D_{A^{m}}\tau\left(  \gamma,\eta\right)  f\right\rangle \right\vert
^{2}=\left\Vert h\right\Vert _{L^{2}\left(  \mathbb{H}\right)  }^{2}.
\]
Furthermore%
\[
h=\sum_{m\in%
\mathbb{Z}
}\sum_{\left(  \gamma,\eta\right)  \in\Lambda}\left\langle h,D_{A^{m}}%
\tau\left(  \gamma,\eta\right)  f\right\rangle D_{A^{m}}\tau\left(
\gamma,\eta\right)  f
\]
with convergence in the $L^{2}$-norm.

We recall that the Plancherel transform (see the section titled
Preliminaries)
\[
\mathbf{P}:L^{2}\left(  \mathbb{H}\right)  \rightarrow\int_{%
\mathbb{R}
^{\ast}}^{\oplus}L^{2}\left(
\mathbb{R}
\right)  \otimes L^{2}\left(
\mathbb{R}
\right)  \left\vert \lambda\right\vert d\lambda
\]
is a unitary operator obtained by extending the group Fourier transform from
$L^{2}\left(  \mathbb{H}\right)  \cap L^{1}\left(  \mathbb{H}\right)  $ to
$L^{2}\left(  \mathbb{H}\right)  .$ Let
\[
\mathbf{H}_{\mathcal{S}}=\mathbf{P}^{-1}\left(  \int_{\mathcal{S}}^{\oplus
}L^{2}\left(
\mathbb{R}
\right)  \otimes L^{2}\left(
\mathbb{R}
\right)  \left\vert \lambda\right\vert d\lambda\right)
\]
such that $\mathcal{S}$ a subset of $%
\mathbb{R}
^{\ast}.$ Then $\mathbf{H}_{\mathcal{S}}$ is a $\tau$-invariant Hilbert
subspace of $L^{2}\left(  \mathbb{H}\right)  $. Here are the main theorems of
this paper which are proved in the third section of this paper.

\begin{theorem}
\label{main}If $\mathcal{S}$ is translation congruent to $\left(  0,1\right]
$ and $\mathcal{S}\subseteq\left[  -1,1\right]  \ $then there is a function
$f\in\mathbf{H}_{\mathcal{S}}$ such that $\tau\left(  \Lambda\right)  f$ is a
Parseval frame in $\mathbf{H}_{\mathcal{S}}$ and
\[
\left\Vert f\right\Vert _{L^{2}\left(  \mathbb{H}\right)  }^{2}\leq\frac{2}%
{3}.
\]

\end{theorem}

Put
\[
\Lambda_{1}=\left\{  \left(  \left[
\begin{array}
[c]{cccc}%
1 & k_{3} & -k_{2} & 0\\
0 & 1 & 0 & k_{2}\\
0 & 0 & 1 & 0\\
0 & 0 & 0 & 1
\end{array}
\right]  ,\left[
\begin{array}
[c]{cccc}%
1 & m_{3} & -m_{2} & 0\\
0 & 1 & 0 & m_{2}\\
0 & 0 & 1 & 0\\
0 & 0 & 0 & 1
\end{array}
\right]  \right)  :k_{i},m_{j}\in\mathbb{Z}\right\}  .
\]
For a given representation $\pi,$ let $\overline{\pi}$ be the corresponding
contragredient representation. Let
\[
\left\{  \pi_{\lambda}:\lambda\in\mathbb{R},\lambda\neq0\right\}
\]
be a parametrizing set for the unitary dual of the Heisenberg group and let
$d\lambda$ be the Lebesgue measure on $%
\mathbb{R}
.$

\begin{theorem}
\label{new}Assume that $\mathcal{S}$ is translation congruent to $\left(
0,1\right]  .$ Let $f\in\mathbf{H}_{\mathcal{S}}.$ If $\tau\left(
\Lambda\right)  f$ is a Parseval frame in $\mathbf{H}_{\mathcal{S}}$\textbf{
}then%
\[
\left\{  \left[  \pi_{\lambda}\left(  \kappa\right)  \otimes\overline{\pi
}_{\lambda}\left(  \eta\right)  \right]  \left(  \mathbf{P}f\right)
(\lambda)\left\vert \lambda\right\vert ^{1/2}:\left(  \kappa,\eta\right)
\in\Lambda_{1}\right\}
\]
is a Parseval frame in $L^{2}\left(
\mathbb{R}
\right)  \otimes L^{2}\left(
\mathbb{R}
\right)  $ for $d\lambda$-almost every $\lambda\in\mathcal{S}$.
\end{theorem}

Let $a,b$ be non-zero real numbers. Let $f\in L^{2}\left(  \mathbb{R}\right)
$. The family of vectors
\[
\mathcal{G}\left(  f,\mathcal{L}\right)  =\left\{  e^{2\pi i\left\langle
k,x\right\rangle }f\left(  x-n\right)  :k\in b\mathbb{Z},n\in a\mathbb{Z}%
\right\}
\]
is called a Gabor system in $L^{2}\left(  \mathbb{R}\right)  .$ For any
function $f\in L^{2}\left(  \mathbb{R}\right)  ,$ let $\overline{f}$ be the
complex conjugate of $f.$

\begin{theorem}
\label{last}Assume that $\mathcal{S}$ is dilation congruent to $\left[
-1,-1/2\right)  \cup\left(  1/2,1\right]  $, translation congruent to $\left(
0,1\right]  $ and that $\mathcal{S}\subseteq\left[  -1,1\right]  .$ Let
$f\in\mathbf{H}_{\mathcal{S}}$ be defined as follows: $\mathbf{P}f\left(
\lambda\right)  =u_{\lambda}\otimes v_{\lambda}\ $such that $\mathcal{G}%
\left(  \left\vert \lambda\right\vert ^{1/4}u_{\lambda},%
\mathbb{Z}
\times\lambda%
\mathbb{Z}
\right)  $ is a Parseval Gabor frame for $d\lambda$-almost every $\lambda
\in\mathcal{S}$, and $\mathcal{G}\left(  \left\vert \lambda\right\vert
^{1/4}\overline{v}_{\lambda},%
\mathbb{Z}
\times\lambda%
\mathbb{Z}
\right)  $ is a Parseval Gabor frame for $d\lambda$-almost every $\lambda
\in\mathcal{S}$. Then, the system
\[
\left\{  D_{A^{m}}\tau\left(  \gamma,\eta\right)  f:m\in%
\mathbb{Z}
,\left(  \gamma,\eta\right)  \in\Lambda\right\}
\]
is a Parseval frame in $L^{2}\left(  \mathbb{H}\right)  .$
\end{theorem}

\begin{example}
Let us suppose that $a=b=\sqrt{2}$ so that
\[
A=\left[
\begin{array}
[c]{cccc}%
2 & 0 & 0 & 0\\
0 & \sqrt{2} & 0 & 0\\
0 & 0 & \sqrt{2} & 0\\
0 & 0 & 0 & 1
\end{array}
\right]  .
\]
Put
\[
\mathcal{S}=\left[  -1,-\frac{1}{2}\right)  \cup\left(  \frac{1}{2},1\right]
.
\]
Then $\mathcal{S}$ is up to a null set translation congruent to $\left(
0,1\right]  .$ Define $f$ such that
\[
\mathbf{P}f(\lambda)=\left\vert \lambda\right\vert ^{1/4}\chi_{\left[
0,1\right)  }\otimes\left\vert \lambda\right\vert ^{1/4}\chi_{\left[
0,1\right)  }%
\]
where $\chi_{\left[  0,1\right)  }$ is the characteristic function of the set
$\left[  0,1\right)  .$ According to Proposition $3.1,$ \cite{Pfander} it is
not hard to check that
\[
\mathcal{G}\left(  \left\vert \lambda\right\vert ^{1/2}\chi_{\left[
0,1\right)  },%
\mathbb{Z}
\times\lambda%
\mathbb{Z}
\right)
\]
is Parseval frame in $L^{2}\left(
\mathbb{R}
\right)  $ for every $\lambda\in\mathcal{S}$. Therefore, appealing to Theorem
\ref{main} and Theorem \ref{last}, $\tau\left(  \Lambda\right)  f$ is a
Parseval frame in $\mathbf{H}_{\mathcal{S}}$ and
\[
\left\{  D_{A^{m}}\tau\left(  \gamma,\eta\right)  f:m\in%
\mathbb{Z}
,\left(  \gamma,\eta\right)  \in\Lambda\right\}
\]
is a Parseval frame in $L^{2}\left(  \mathbb{H}\right)  .$ Next, put
\[
M\left(  x,y,z\right)  =\left[
\begin{array}
[c]{cccc}%
1 & x & -y & z\\
0 & 1 & 0 & y\\
0 & 0 & 1 & 0\\
0 & 0 & 0 & 1
\end{array}
\right]  .
\]
Using the group Fourier inverse transform provided in Theorem $4.15$
\cite{Fuhr cont}, we obtain
\[
f\left(  M\left(  x,y,z\right)  \right)  =\int_{\mathcal{S}}\left\langle
\chi_{\left[  0,1\right)  },\pi_{\lambda}\left(  x\right)  \chi_{\left[
0,1\right)  }\right\rangle \left\vert \lambda\right\vert ^{3/2}d\lambda.
\]
Therefore,
\[
f\left(  M\left(  x,y,z\right)  \right)  =\left\{
\begin{array}
[c]{cc}%
\int_{\mathcal{S}}\frac{\left(  \exp\left(  2\pi i\lambda\left(  y-z\right)
\right)  -\exp\left(  2\pi i\lambda\left(  yx-z\right)  \right)  \right)
\left\vert \lambda\right\vert ^{3/2}}{2\pi i\lambda y}d\lambda & \text{ if
}x\in\left[  0,1\right)  \text{ and }y\neq0\\
\frac{\left(  8-\sqrt{2}\right)  \left(  1-x\right)  }{10} & \text{if }%
x\in\left[  0,1\right)  \text{ and }y=0\\
\int_{\mathcal{S}}\frac{\left(  \exp\left(  2\pi i\lambda\left(  y\left(
x+1\right)  -z\right)  \right)  -\exp\left(  -2\pi i\lambda z\right)  \right)
\left\vert \lambda\right\vert ^{3/2}}{2\pi i\lambda y}d\lambda & \text{if
}x\in\left(  -1,0\right]  \text{ and }y\neq0\\
\frac{\left(  8-\sqrt{2}\right)  \left(  1+x\right)  }{10} & \text{if }%
x\in\left(  -1,0\right]  \text{ and }y=0\\
0 & \left\vert x\right\vert \geq1
\end{array}
\right.  .
\]

\end{example}

\section{Preliminaries}

\subsection{Notations and definitions}

The punctured line which is the set of all non-zero real numbers is denoted $%
\mathbb{R}
^{\ast}.$ The general linear group which consists of invertible real matrices
of order $d$ is denoted $GL\left(  d,%
\mathbb{R}
\right)  .$ Let $G$ and $H$ be two groups. If $G$ and $H$ are isomorphic, we
write $G\cong H.$ Let $T$ be a linear operator defined on some Hilbert space.
The adjoint of $T$ is denoted $T^{\ast}.$ All sets of interest in this paper
should be assumed to be measurable, and we shall identify subsets whose
symmetric difference has Lebesgue measure zero. For example, we make no
distinction between $\left[  0,1\right]  $ and $\left(  0,1\right]  .$ Also,
all functions mentioned in this paper should be assumed to be measurable functions.

For subsets $I$ and $J$ of $\mathbb{R},$ we say that $I$ and $J$ are
translation congruent if there is a bijection $\rho:I\rightarrow J$ and an
integer valued-function $k$ on $I$ such that $\rho\left(  \lambda\right)
=\lambda+k\left(  \lambda\right)  .$ For example, if $I$ is translation
congruent to $\left[  0,1\right)  $ then $I$ tiles the real line by
$\mathbb{Z}.$ Next, we say that $I$ and $J$ are dilation congruent if there
exists a bijection $\delta:I\rightarrow J$ and an integer-valued function $j$
on $I$ such that $\delta\left(  \lambda\right)  =2^{j\left(  \lambda\right)
}\lambda.$

\subsection{Plancherel theory}

The facts presented in this subsection are pretty standard. We refer the
interested reader to Chapter $7,$ \cite{Folland}, Chapter $2$ \cite{Thang} and
Chapter $4,$ \cite{Corwin}.

Let
\[
\mathbb{P}=\left\{  \left[
\begin{array}
[c]{cccc}%
1 & 0 & -y & z\\
0 & 1 & 0 & y\\
0 & 0 & 1 & 0\\
0 & 0 & 0 & 1
\end{array}
\right]  :z,y\in%
\mathbb{R}
\right\}
\]
be a maximal abelian subgroup of the Heisenberg group. Then the Heisenberg
group is isomorphic to a semi-direct product of the type
\[
\mathbb{P\rtimes}\left\{  \left[
\begin{array}
[c]{cccc}%
1 & x & 0 & 0\\
0 & 1 & 0 & 0\\
0 & 0 & 1 & 0\\
0 & 0 & 0 & 1
\end{array}
\right]  :x\in%
\mathbb{R}
\right\}  \text{ }%
\]
which is also isomorphic to $\mathbb{P\rtimes}%
\mathbb{R}
.$ For each $\lambda\in%
\mathbb{R}
,$ we define a corresponding character $\chi_{\lambda}$ on $\mathbb{P}$ by
\[
\chi_{\lambda}\left(  \left[
\begin{array}
[c]{cccc}%
1 & 0 & -y & z\\
0 & 1 & 0 & y\\
0 & 0 & 1 & 0\\
0 & 0 & 0 & 1
\end{array}
\right]  \right)  =e^{-2\pi i\lambda\left(  z\right)  }.
\]
According to the theory of Mackey, or the orbit method (see \cite{Corwin}) it
is not too hard to show that the unitary dual of $\mathbb{H}$ which we denote
by $\widehat{\mathbb{H}}$ is up to a null set equal to
\[
\left\{  \pi_{\lambda}=\mathrm{Ind}_{\mathbb{P}}^{\mathbb{H}}\left(
\chi_{\lambda}\right)  :\lambda\in%
\mathbb{R}
^{\ast}\right\}  .
\]
In fact, there are two families of unitary irreducible representations of the
Heisenberg group. The first family of irreducible representations only
contains characters and forms a set of Plancherel measure zero, and is
therefore negligible. The second family of irreducible representations are
infinite-dimensional representations which are parametrized by the punctured
line as follows: $\lambda\mapsto\pi_{\lambda}=\mathrm{Ind}_{\mathbb{P}%
}^{\mathbb{H}}\left(  \chi_{\lambda}\right)  .$ The reader who is not familiar
with the theory of induced representation is invited to refer to the book of
Folland \cite{Folland}. Based on properties of induced representations, each
unitary representation $\pi_{\lambda}$ is realized as acting in the Hilbert
space of square integrable functions defined over $\mathbb{H}/\mathbb{P}.$
More precisely, $\pi_{\lambda}$ acts in $L^{2}\left(  \mathbb{H}%
/\mathbb{P}\right)  $ which we naturally identify with $L^{2}\left(
\mathbb{R}
\right)  $ such that for $f\in L^{2}\left(
\mathbb{R}
\right)  ,$
\[
\pi_{\lambda}\left(  \left[
\begin{array}
[c]{cccc}%
1 & 0 & 0 & z\\
0 & 1 & 0 & 0\\
0 & 0 & 1 & 0\\
0 & 0 & 0 & 1
\end{array}
\right]  \right)  f\left(  t\right)  =e^{2\pi i\lambda\left(  z\right)
}f\left(  t\right)  ,
\]%
\[
\pi_{\lambda}\left(  \left[
\begin{array}
[c]{cccc}%
1 & 0 & -y & 0\\
0 & 1 & 0 & y\\
0 & 0 & 1 & 0\\
0 & 0 & 0 & 1
\end{array}
\right]  \right)  f\left(  t\right)  =e^{-2\pi i\lambda yt}f\left(  t\right)
\]
and
\[
\pi_{\lambda}\left(  \left[
\begin{array}
[c]{cccc}%
1 & 0 & x & 0\\
0 & 1 & 0 & 0\\
0 & 0 & 1 & 0\\
0 & 0 & 0 & 1
\end{array}
\right]  \right)  f\left(  t\right)  =f\left(  t-x\right)  .
\]

Let $\mathcal{F}$ be the operator-valued Fourier transform defined on
$L^{2}(\mathbb{H})\cap L^{1}(\mathbb{H})$ by $\mathcal{F}\left(  f\right)
\left(  \lambda\right)  =\int_{\mathbb{H}}\pi_{\lambda}\left(  n\right)
f\left(  n\right)  dn.$ The Fourier transform can also be defined as follows:
for any vectors $u,v\in L^{2}\left(
\mathbb{R}
\right)  ,$%
\[
\left\langle \mathcal{F}\left(  f\right)  \left(  \lambda\right)
u,v\right\rangle =\int_{\mathbb{H}}f\left(  n\right)  \left\langle
\pi_{\lambda}\left(  n\right)  u,v\right\rangle dn.
\]
Since the $\pi_{\lambda}\left(  n\right)  $ are unitary operators, then
\[
\left\vert \left\langle \pi_{\lambda}\left(  n\right)  u,v\right\rangle
\right\vert \leq\left\Vert u\right\Vert _{L^{2}\left(
\mathbb{R}
\right)  }\left\Vert v\right\Vert _{L^{2}\left(
\mathbb{R}
\right)  }.
\]
Thus, for $f\in$ $L^{2}(\mathbb{H})\cap L^{1}(\mathbb{H})$
\begin{align*}
\left\vert \left\langle \mathcal{F}\left(  f\right)  \left(  \lambda\right)
u,v\right\rangle \right\vert  &  \leq\int_{\mathbb{H}}\left\vert f\left(
n\right)  \left\langle \pi_{\lambda}\left(  n\right)  u,v\right\rangle
\right\vert dn\\
&  \leq\left\Vert u\right\Vert _{L^{2}\left(
\mathbb{R}
\right)  }\left\Vert v\right\Vert _{L^{2}\left(
\mathbb{R}
\right)  }\left\Vert f\right\Vert _{L^{1}(\mathbb{H})}.
\end{align*}
Therefore, $\mathcal{F}\left(  f\right)  \left(  \lambda\right)  $ is a
bounded operator on $L^{2}\left(
\mathbb{R}
\right)  $. For $f\in$ $L^{2}(\mathbb{H})\cap L^{1}(\mathbb{H}),$ it can also
be shown (see Chapter $2,$ \cite{Thang}) that $\mathcal{F}\left(  f\right)
\left(  \lambda\right)  $ is actually a Hilbert-Schmidt operator and that the
Fourier transform $\mathcal{F}$ extends to the Plancherel transform:
\[
\mathbf{P}:L^{2}\left(  \mathbb{H}\right)  \rightarrow\int_{%
\mathbb{R}
^{\ast}}^{\oplus}L^{2}\left(  \mathbb{R}\right)  \otimes L^{2}\left(
\mathbb{R}\right)  \text{ }\left\vert \lambda\right\vert d\lambda
\]
such that for any $f\in L^{2}\left(  \mathbb{H}\right)  ,$%
\begin{equation}
\left\Vert f\right\Vert _{L^{2}\left(  \mathbb{H}\right)  }^{2}=\int_{%
\mathbb{R}
^{\ast}}\left\Vert \mathbf{P}f\left(  \lambda\right)  \right\Vert
_{\mathcal{HS}}^{2}\text{ }\left\vert \lambda\right\vert d\lambda.
\label{unitary}%
\end{equation}
Let $d\lambda$ be the Lebesgue measure on $%
\mathbb{R}
^{\ast}.$ The Plancherel measure for this group is supported on the punctured
line $%
\mathbb{R}
^{\ast}$ and is the weighted Lebesgue measure $\left\vert \lambda\right\vert
d\lambda.$ $||\cdot||_{\mathcal{HS}}$ denotes the Hilbert-Schmidt norm on
$L^{2}\left(  \mathbb{R}\right)  \otimes L^{2}\left(  \mathbb{R}\right)  $ and
clearly (\ref{unitary}) implies that $\mathbf{P}$ is a unitary transform.

We recall that given $T,P\in L^{2}\left(  \mathbb{R}\right)  \otimes
L^{2}\left(  \mathbb{R}\right)  ,$%
\[
\left\langle T,P\right\rangle _{\mathcal{HS}}=%
{\displaystyle\sum\limits_{k}}
\left\langle Te_{k},Pe_{k}\right\rangle
\]
where $\left\{  e_{k}\right\}  _{k}$ is an orthonormal basis for $L^{2}\left(
\mathbb{R}\right)  .$ Also, given a Hilbert-Schmidt operator $T:L^{2}\left(
\mathbb{R}\right)  \rightarrow L^{2}\left(  \mathbb{R}\right)  ,$ we say that
$T$ is a finite-rank operator if the range of $T$ is a finite dimensional
subspace of $L^{2}\left(  \mathbb{R}\right)  .$ It is well-known that
finite-rank operators form a dense subspace of $L^{2}\left(  \mathbb{R}%
\right)  \otimes L^{2}\left(  \mathbb{R}\right)  .$ Moreover, if $T$ is a
rank-one operator, then $T=u\otimes v$ for some $u,v\in L^{2}\left(
\mathbb{R}\right)  ,$ and the inner product of arbitrary rank-one operators in
$L^{2}\left(  \mathbb{R}\right)  \otimes L^{2}\left(  \mathbb{R}\right)  $ is
given by%
\[
\left\langle u\otimes v,w\otimes y\right\rangle _{\mathcal{HS}}=\left\langle
u,w\right\rangle _{L^{2}\left(  \mathbb{R}\right)  }\left\langle
v,y\right\rangle _{L^{2}\left(  \mathbb{R}\right)  }.
\]

Let $\tau$ be the biregular representation $\tau:\mathbb{H\times H}%
\rightarrow\mathcal{U}\left(  L^{2}\left(  \mathbb{H}\right)  \right)  $ (See
Page $233,$ \cite{Folland}) defined by
\[
\tau\left(  u,v\right)  f=L\left(  u\right)  R\left(  v\right)  f
\]
where $L\left(  u\right)  f\left(  x\right)  =f\left(  u^{-1}x\right)  $ and
$R\left(  v\right)  f\left(  x\right)  =f\left(  xv\right)  .$ The Plancherel
transform intertwines the biregular representation with a direct integral of
tensor representations as follows:
\[
\mathbf{P}\circ\tau\left(  x,y\right)  \circ\mathbf{P}^{-1}=\int_{%
\mathbb{R}
^{\ast}}^{\oplus}\pi_{\lambda}\left(  x\right)  \otimes\overline{\pi}%
_{\lambda}\left(  y\right)  \text{ }\left\vert \lambda\right\vert d\lambda.
\]
Furthermore, for $\lambda\in%
\mathbb{R}
^{\ast}$
\[
\mathbf{P}(\tau\left(  x,y\right)  \phi)(\lambda)=\pi_{\lambda}(x)\left(
\mathbf{P}\phi\right)  (\lambda)\left(  \overline{\pi}_{\lambda}\left(
y\right)  \right)  ^{\ast}=\left[  \pi_{\lambda}\left(  x\right)
\otimes\overline{\pi}_{\lambda}\left(  y\right)  \right]  \left(
\mathbf{P}\phi\right)  (\lambda)
\]
with $\overline{\pi}_{\lambda}$ being the contragredient of the representation
$\pi_{\lambda}$ acting in the dual of $L^{2}\left(
\mathbb{R}
\right)  .$ We remind the reader that
\[
\overline{\pi}_{\lambda}\left(  x\right)  =\pi_{\lambda}\left(  x^{-1}\right)
^{tr}%
\]
where $tr$ denotes the transpose of an operator.

\subsection{A class of dilations}

Now, we will present a class of dilations which will be important in the
construction of wavelets. Also, in the proof of Lemma \ref{cyclic}, we will
clarify why this particular realization of the Heisenberg group as a subgroup
of $GL\left(  4,%
\mathbb{R}
\right)  $ is convenient. In fact, this finite-dimensional representation of
the Heisenberg group allows us to define a natural automorphism of the
Heisenberg group (obtained by conjugations) which generates a class of
dilations. Put
\[
A=\left[
\begin{array}
[c]{cccc}%
\frac{bc}{d} & 0 & 0 & 0\\
0 & b & 0 & 0\\
0 & 0 & c & 0\\
0 & 0 & 0 & d
\end{array}
\right]  \in GL\left(  4,%
\mathbb{R}
\right)
\]
where $b,c,d\in%
\mathbb{R}
^{\ast}.$ Let $\varphi_{A}:\mathbb{H\rightarrow H}$ such that
\[
\varphi_{A}\left(  M\right)  =AMA^{-1}.
\]

\begin{lemma}
\label{cyclic}$\left\langle \varphi_{A}\right\rangle $ is a subgroup of
$\mathrm{Aut}\left(  \mathbb{H}\right)  .$
\end{lemma}

\begin{proof}
Let
\[
X=\left[
\begin{array}
[c]{cccc}%
a & 0 & 0 & 0\\
0 & b & 0 & 0\\
0 & 0 & c & 0\\
0 & 0 & 0 & d
\end{array}
\right]  .
\]
With some simple calculations, it is easy to see that
\begin{equation}
X\left[
\begin{array}
[c]{cccc}%
1 & x & -y & z\\
0 & 1 & 0 & y\\
0 & 0 & 1 & 0\\
0 & 0 & 0 & 1
\end{array}
\right]  X^{-1}=\left[
\begin{array}
[c]{cccc}%
1 & \frac{a}{b}x & -\frac{a}{c}y & \frac{a}{d}z\\
0 & 1 & 0 & \frac{b}{d}y\\
0 & 0 & 1 & 0\\
0 & 0 & 0 & 1
\end{array}
\right]  .\label{conj}%
\end{equation}
Therefore, (\ref{conj}) induces an action of the group generated by $X$ on
$\mathbb{H}$ if and only if $\frac{b}{d}=\frac{a}{c}.$ Next, solving the above
equation for $a,$ we obtain%
\[
X=A=\left[
\begin{array}
[c]{cccc}%
\frac{bc}{d} & 0 & 0 & 0\\
0 & b & 0 & 0\\
0 & 0 & c & 0\\
0 & 0 & 0 & d
\end{array}
\right]  .
\]
Now, to show that $A$ defines an automorphism, we observe that
\[
A\left(  \left[
\begin{array}
[c]{cccc}%
1 & x_{1} & -y_{1} & z_{1}\\
0 & 1 & 0 & y_{1}\\
0 & 0 & 1 & 0\\
0 & 0 & 0 & 1
\end{array}
\right]  \left[
\begin{array}
[c]{cccc}%
1 & x_{2} & -y_{2} & z_{2}\\
0 & 1 & 0 & y_{2}\\
0 & 0 & 1 & 0\\
0 & 0 & 0 & 1
\end{array}
\right]  \right)  A^{-1}%
\]
is equal to
\begin{equation}
\left[
\begin{array}
[c]{cccc}%
1 & \frac{c}{d}\left(  x_{1}+x_{2}\right)   & -\frac{b}{d}\left(  y_{1}%
+y_{2}\right)   & b\frac{c}{d^{2}}\left(  z_{1}+z_{2}+x_{1}y_{2}\right)  \\
0 & 1 & 0 & \frac{b}{d}\left(  y_{1}+y_{2}\right)  \\
0 & 0 & 1 & 0\\
0 & 0 & 0 & 1
\end{array}
\right]  .\label{pr}%
\end{equation}
Finally,%
\[
\left(  A\left[
\begin{array}
[c]{cccc}%
1 & x_{1} & -y_{1} & z_{1}\\
0 & 1 & 0 & y_{1}\\
0 & 0 & 1 & 0\\
0 & 0 & 0 & 1
\end{array}
\right]  A^{-1}\right)  \left(  A\left[
\begin{array}
[c]{cccc}%
1 & x_{2} & -y_{2} & z_{2}\\
0 & 1 & 0 & y_{2}\\
0 & 0 & 1 & 0\\
0 & 0 & 0 & 1
\end{array}
\right]  A^{-1}\right)
\]
is equal to (\ref{pr}) as well. Thus conjugation by $A$ defines a bijective
homomorphism of the Heisenberg group. 
\end{proof}

Now, we will define the dilation action which will play an essential role in
the construction of discrete wavelets in $L^{2}\left(  \mathbb{H}\right)  $.
Let $\left(  a,b\right)  \in%
\mathbb{R}
^{\ast}\times%
\mathbb{R}
^{\ast}$ and define
\[
A_{\left(  a,b\right)  }=A=\left[
\begin{array}
[c]{cccc}%
ab & 0 & 0 & 0\\
0 & b & 0 & 0\\
0 & 0 & a & 0\\
0 & 0 & 0 & 1
\end{array}
\right]  \in GL\left(  4,%
\mathbb{R}
\right)
\]
so that
\[
A\left[
\begin{array}
[c]{cccc}%
1 & x & -y & z\\
0 & 1 & 0 & y\\
0 & 0 & 1 & 0\\
0 & 0 & 0 & 1
\end{array}
\right]  A^{-1}=\left[
\begin{array}
[c]{cccc}%
1 & ax & -by & abz\\
0 & 1 & 0 & by\\
0 & 0 & 1 & 0\\
0 & 0 & 0 & 1
\end{array}
\right]  .
\]
Moreover, we assume that $\left(  a,b\right)  \in%
\mathbb{R}
^{\ast}\times%
\mathbb{R}
^{\ast}$ are chosen so that
\begin{equation}
ab=2.\label{con}%
\end{equation}
We acknowledge that (\ref{con}) seems rather peculiar at first. However, the
need to impose this condition will be clarified in the proof of Theorem
\ref{main} and Theorem \ref{last}. Moreover, we also observe that in general
the group $A^{-1}\Gamma A$ is not a subgroup of $\Gamma.$ In fact,
$A^{-1}\Gamma A$ is not generally even a discrete subgroup of the Heisenberg
group. Therefore, we remark that the structure that we are interested in this
paper does not fit the definition of affine structure given in \cite{Baggett}.

Put
\[
H=\left\langle A\right\rangle \cong%
\mathbb{Z}
\]
and define $D:H\rightarrow\mathcal{U}\left(  L^{2}\left(  \mathbb{H}\right)
\right)  $ such that
\[
D_{A}f\left(  n\right)  =\left\vert \det A\right\vert ^{-1/2}f\left(
A^{-1}nA\right)  =\frac{1}{2}f\left(  A^{-1}nA\right)  .
\]
Then $D$ is a dilation representation which will play an important role in the
construction of bandlimited wavelets on the Heisenberg group.

\subsection{Short survey on frame theory}

Given a countable sequence $\left\{  f_{i}\right\}  _{i\in I}$ of vectors in a
Hilbert space $\mathbf{H},$ we say $\left\{  f_{i}\right\}  _{i\in I}$ forms a
frame if and only if there exist strictly positive real numbers $A,B$ such
that for any vector $f\in\mathbf{H}$
\[
A\left\Vert f\right\Vert ^{2}\leq\sum_{i\in I}\left\vert \left\langle
f,f_{i}\right\rangle \right\vert ^{2}\leq B\left\Vert f\right\Vert ^{2}.
\]
In the case where $A=B$, the sequence of vectors $\left\{  f_{i}\right\}
_{i\in I}$ forms what we call a tight frame, and if $A=B=1$, $\left\{
f_{i}\right\}  _{i\in I}$ is called a Parseval frame\textbf{ }or\textbf{ }a
normalized tight frame. Let us suppose that $\left\{  f_{i}\right\}  _{i\in
I}$ is a Parseval frame in $\mathbf{H.}$ Then for any vector $h\in\mathbf{H,}$
we have the following remarkable expansion formula:
\[
h=\sum_{i\in I}\left\langle h,f_{i}\right\rangle f_{i}.
\]
A lattice $\mathcal{L}$ in $\mathbb{R}^{2d}$ is a discrete additive subgroup
of $\mathbb{R}^{2d}$. A lattice $\mathcal{L}$ is called a full-rank lattice if
$\mathcal{L}=M\mathbb{Z}^{2d}$ for some invertible matrix $M$ of order $2d$.
We say $\mathcal{L}$ is separable if $\mathcal{L}=A\mathbb{Z}^{d}\times
B\mathbb{Z}^{d}$ and $A,B$ are invertible matrices of order $d.$ Let $f\in
L^{2}\left(  \mathbb{R}^{d}\right)  $. The family of functions
\[
\mathcal{G}\left(  f,\mathcal{L}\right)  =\left\{  e^{2\pi i\left\langle
k,x\right\rangle }f\left(  x-n\right)  :k\in B\mathbb{Z}^{d},n\in
A\mathbb{Z}^{d}\right\}
\]
is called a Gabor system. The volume of $\mathcal{L}=M\mathbb{Z}^{2d}$ is
defined as $\mathrm{vol}\left(  \mathcal{L}\right)  =\left\vert \det
M\right\vert $ and the density of $\mathcal{L}$ is defined as $d\left(
\mathcal{L}\right)  =\left\vert \det M\right\vert ^{-1}.$

\begin{lemma}
\label{density}Given a separable full-rank lattice $\mathcal{L}=A\mathbb{Z}%
^{d}\times B\mathbb{Z}^{d}$ in $\mathbb{R}^{2d}$, the following statements are equivalent.
\end{lemma}

\begin{enumerate}
\item There exists $f\in L^{2}(\mathbb{R}^{d})$ such that $\mathcal{G}\left(
f,\mathcal{L}\right)  $ is a Parseval frame in $L^{2}\left(  \mathbb{R}%
^{d}\right)  .$

\item $\mathrm{vol}\left(  \mathcal{L}\right)  =\left\vert \det A\det
B\right\vert \leq1.$

\item There exists $f\in L^{2}\left(  \mathbb{R}^{d}\right)  $ such that
$\mathcal{G}\left(  f,\mathcal{L}\right)  $ is complete in $L^{2}\left(
\mathbb{R}^{d}\right)  .$
\end{enumerate}

See theorem $3.3$ in \cite{Han Yang Wang}.

\begin{lemma}
\label{ONB copy(1)} Let $\mathcal{L}$ be a full rank lattice in $\mathbb{R}%
^{2d}$. If $\mathcal{G}\left(  f,\mathcal{L}\right)  $ is a Parseval frame for
$L^{2}(\mathbb{R}^{d})$, then $\Vert f\Vert^{2}=\mathrm{vol}(\mathcal{L}).$
\end{lemma}

For a complete proof of the Lemma above, we refer the reader to \cite{Han Yang
Wang}. The following lemma is due to Khosravi and Asgari (see Theorem $2.3$
\cite{tensor})

\begin{lemma}
\label{tensor}Let $\left\{  x_{n}\right\}  _{n\in I}$ and $\left\{
y_{m}\right\}  _{m\in J}$ be two Parseval frames for Hilbert spaces $H$ and
$K,$ respectively. Then $\left\{  x_{n}\otimes y_{m}\right\}  _{\left(
n,m\right)  \in I\times J}$ is a Parseval frame for the Hilbert space
$H\otimes K.$
\end{lemma}

\section{Proofs of the main results}

\begin{lemma}
\label{LEM} Let $v\in L^{2}\left(
\mathbb{R}
\right)  ,$ and $\overline{v}$ its conjugate. If $\mathcal{G}\left(
\overline{v},%
\mathbb{Z}
\times\lambda%
\mathbb{Z}
\right)  $ is a Parseval Gabor frame then
\[
\sum_{\eta\in\mathbb{I}}\left\vert \left\langle \overline{\pi}_{\lambda
}\left(  \eta\right)  v,u\right\rangle _{L^{2}\left(
\mathbb{R}
\right)  }\right\vert ^{2}=\left\Vert u\right\Vert _{L^{2}\left(
\mathbb{R}
\right)  }^{2}%
\]
for all $u\in L^{2}\left(
\mathbb{R}
\right)  $ and
\[
\mathbb{I}=\left\{  \left[
\begin{array}
[c]{cccc}%
1 & m_{3} & -m_{2} & 0\\
0 & 1 & 0 & m_{2}\\
0 & 0 & 1 & 0\\
0 & 0 & 0 & 1
\end{array}
\right]  :\left(  m_{2},m_{3}\right)  \in%
\mathbb{Z}
^{2}\right\}  .
\]

\end{lemma}

\begin{proof}
Let us suppose that $\mathcal{G}\left(  \overline{v},%
\mathbb{Z}
\times\lambda%
\mathbb{Z}
\right)  $ is a Parseval Gabor frame. Now, let
\[
v=\sum_{k\in\mathbb{J}}\alpha_{k}e_{k}\text{ and }u=\sum_{j\in\mathbb{J}}%
\beta_{j}e_{j},\text{ }\alpha_{k},\beta_{j}\in%
\mathbb{C}
\]
where $\left\{  e_{k}:k\in\mathbb{J}\right\}  $ is an orthonormal basis for
$L^{2}\left(
\mathbb{R}
\right)  $ such that $\overline{e_{k}}=e_{k}$ for all $k\in\mathbb{J}$. Then
\begin{align*}
\sum_{\eta\in\mathbb{I}}\left\vert \left\langle \overline{\pi}_{\lambda
}\left(  \eta\right)  v,u\right\rangle _{L^{2}\left(
\mathbb{R}
\right)  }\right\vert ^{2}  &  =\sum_{\eta\in\mathbb{I}}\left\vert
\left\langle \overline{\pi}_{\lambda}\left(  \eta\right)  \left(  \overset
{v}{\overbrace{\sum_{k\in\mathbb{J}}\alpha_{k}e_{k}}}\right)  ,\overset
{u}{\overbrace{\sum_{j\in\mathbb{J}}\beta_{j}e_{j}}}\right\rangle \right\vert
^{2}\\
&  =\sum_{\eta\in\mathbb{I}}\left\vert \sum_{k\in\mathbb{J}}\sum
_{j\in\mathbb{J}}\overline{\beta_{j}}\alpha_{k}\left\langle \overline{\pi
}_{\lambda}\left(  \eta\right)  e_{k},e_{j}\right\rangle \right\vert ^{2}.
\end{align*}
Next, since
\[
\overline{\pi}_{\lambda}\left(  \eta\right)  e_{k}=\pi_{\lambda}\left(
\eta^{-1}\right)  ^{tr}e_{k},
\]
then
\begin{align*}
\sum_{\eta\in\mathbb{I}}\left\vert \left\langle \overline{\pi}_{\lambda
}\left(  \eta\right)  v,u\right\rangle \right\vert ^{2}  &  =\sum_{\eta
\in\mathbb{I}}\left\vert \sum_{k\in\mathbb{J}}\sum_{j\in\mathbb{J}}%
\overline{\beta_{j}}\alpha_{k}\left\langle \pi_{\lambda}\left(  \eta
^{-1}\right)  ^{tr}e_{k},e_{j}\right\rangle _{L^{2}\left(
\mathbb{R}
\right)  }\right\vert ^{2}\\
&  =\sum_{\eta\in\mathbb{I}}\left\vert \sum_{k\in\mathbb{J}}\sum
_{j\in\mathbb{J}}\overline{\beta_{j}}\alpha_{k}\left\langle \pi_{\lambda
}\left(  \eta^{-1}\right)  e_{j},e_{k}\right\rangle _{L^{2}\left(
\mathbb{R}
\right)  }\right\vert ^{2}.
\end{align*}
The second equality above is due to the fact that
\[
\left\langle \pi_{\lambda}\left(  \eta\right)  ^{tr}e_{j},e_{k}\right\rangle
_{L^{2}\left(
\mathbb{R}
\right)  }=\left\langle \pi_{\lambda}\left(  \eta\right)  e_{k},e_{j}%
\right\rangle _{L^{2}\left(
\mathbb{R}
\right)  }.
\]
Next,
\begin{align*}
\sum_{\eta\in\mathbb{I}}\left\vert \left\langle \overline{\pi}_{\lambda
}\left(  \eta\right)  v,u\right\rangle _{L^{2}\left(
\mathbb{R}
\right)  }\right\vert ^{2}  &  =\sum_{\eta\in\mathbb{I}}\left\vert
\left\langle \sum_{j\in\mathbb{J}}\pi_{\lambda}\left(  \eta^{-1}\right)
\overline{\beta_{j}}e_{j},\sum_{k\in\mathbb{J}}\overline{\alpha_{k}}%
e_{k}\right\rangle _{L^{2}\left(
\mathbb{R}
\right)  }\right\vert ^{2}\\
&  =\sum_{\eta\in\mathbb{I}}\left\vert \left\langle \sum_{j\in\mathbb{J}%
}\overline{\beta_{j}}e_{j},\sum_{k\in\mathbb{J}}\overline{\alpha_{k}}%
\pi_{\lambda}\left(  \eta\right)  e_{k}\right\rangle _{L^{2}\left(
\mathbb{R}
\right)  }\right\vert ^{2}\\
&  =\sum_{\eta\in\mathbb{I}}\left\vert \left\langle \sum_{j\in\mathbb{J}%
}\overline{\beta_{j}}e_{j},\pi_{\lambda}\left(  \eta\right)  \overset
{=\overline{v}}{\overbrace{\sum_{k\in\mathbb{J}}\overline{\alpha_{k}}e_{k}}%
}\right\rangle _{L^{2}\left(
\mathbb{R}
\right)  }\right\vert ^{2}\\
&  =\sum_{\eta\in\mathbb{I}}\left\vert \left\langle \sum_{j\in\mathbb{J}%
}\overline{\beta_{j}}e_{j},\pi_{\lambda}\left(  \eta\right)  \overline
{v}\right\rangle _{L^{2}\left(
\mathbb{R}
\right)  }\right\vert ^{2}\\
&  =\left\Vert \sum_{j\in\mathbb{J}}\overline{\beta_{j}}e_{j}\right\Vert
_{L^{2}\left(
\mathbb{R}
\right)  }^{2}\\
&  =\left\Vert u\right\Vert _{L^{2}\left(
\mathbb{R}
\right)  }^{2}.
\end{align*}

\end{proof}

\subsection{Proof of Theorem \ref{main}}

Let
\[
\left(  \gamma,\eta\right)  =\left(  \left[
\begin{array}
[c]{cccc}%
1 & k_{3} & -k_{2} & k_{1}\\
0 & 1 & 0 & k_{2}\\
0 & 0 & 1 & 0\\
0 & 0 & 0 & 1
\end{array}
\right]  ,\left[
\begin{array}
[c]{cccc}%
1 & m_{3} & -m_{2} & 0\\
0 & 1 & 0 & m_{2}\\
0 & 0 & 1 & 0\\
0 & 0 & 0 & 1
\end{array}
\right]  \right)  \in\Lambda
\]
where $k_{i}$ and $m_{j}$ are integers. Also, we define
\begin{equation}
\Lambda_{1}=\left\{  \left(  \left[
\begin{array}
[c]{cccc}%
1 & k_{3} & -k_{2} & 0\\
0 & 1 & 0 & k_{2}\\
0 & 0 & 1 & 0\\
0 & 0 & 0 & 1
\end{array}
\right]  ,\left[
\begin{array}
[c]{cccc}%
1 & m_{3} & -m_{2} & 0\\
0 & 1 & 0 & m_{2}\\
0 & 0 & 1 & 0\\
0 & 0 & 0 & 1
\end{array}
\right]  \right)  :k_{i},m_{j}\in%
\mathbb{Z}
\right\}  .\label{above}%
\end{equation}
We will show that there is a function $f$ such that for any vector
$h\in\mathbf{H}_{\mathcal{S}}\mathbf{,}$
\[%
{\displaystyle\sum\limits_{\left(  \gamma,\eta\right)  \in\Lambda}}
\left\vert \left\langle h,\tau\left(  \gamma,\eta\right)  f\right\rangle
_{L^{2}\left(  \mathbb{H}\right)  }\right\vert ^{2}=\left\Vert h\right\Vert
_{L^{2}\left(  \mathbb{H}\right)  }^{2}.
\]
Using the fact that $\mathbf{P}$ is a unitary map, and that
\[
\mathbf{P}(\tau\left(  x,y\right)  \phi)(\lambda)=\pi_{\lambda}(x)\left(
\mathbf{P}\phi\right)  (\lambda)\left(  \overline{\pi}_{\lambda}\left(
y\right)  \right)  ^{\ast}=\left[  \pi_{\lambda}\left(  x\right)
\otimes\overline{\pi}_{\lambda}\left(  y\right)  \right]  \left(
\mathbf{P}\phi\right)  (\lambda),
\]
we have
\[%
{\displaystyle\sum\limits_{\left(  \gamma,\eta\right)  \in\Lambda}}
\left\vert \left\langle h,\tau\left(  \gamma,\eta\right)  f\right\rangle
_{L^{2}\left(  \mathbb{H}\right)  }\right\vert ^{2}=%
{\displaystyle\sum\limits_{\left(  \gamma,\eta\right)  \in\Lambda}}
\left\vert \int_{\mathcal{S}}\left\langle \mathbf{P}h(\lambda),\left[
\pi_{\lambda}\left(  \gamma\right)  \otimes\overline{\pi}_{\lambda}\left(
\eta\right)  \right]  \mathbf{P}f(\lambda)\left\vert \lambda\right\vert
\right\rangle _{\mathcal{HS}}d\lambda\right\vert ^{2}.
\]
Next, we define
\[
F_{\kappa,\eta}\left(  \lambda\right)  =\left\langle \mathbf{P}h(\lambda
),\left[  \pi_{\lambda}\left(  \kappa\right)  \otimes\overline{\pi}_{\lambda
}\left(  \eta\right)  \right]  \mathbf{P}f(\lambda)\left\vert \lambda
\right\vert \right\rangle _{\mathcal{HS}}.
\]
Using (\ref{above}), we obtain:%
\begin{align*}
&
{\displaystyle\sum\limits_{\left(  \gamma,\eta\right)  \in\Lambda}}
\left\vert \left\langle h,\tau\left(  \gamma,\eta\right)  f\right\rangle
_{L^{2}\left(  \mathbb{H}\right)  }\right\vert ^{2}\\
&  =%
{\displaystyle\sum\limits_{\left(  \kappa,\eta\right)  \in\Lambda_{1}}}
\sum_{k_{1}\in%
\mathbb{Z}
}\left\vert \int_{\mathcal{S}}e^{2\pi i\left\langle \lambda,k_{1}\right\rangle
}\overset{=F_{\kappa,\eta}\left(  \lambda\right)  }{\overbrace{\left\langle
\mathbf{P}h(\lambda),\left[  \pi_{\lambda}\left(  \kappa\right)
\otimes\overline{\pi}_{\lambda}\left(  \eta\right)  \right]  \mathbf{P}%
f(\lambda)\left\vert \lambda\right\vert \right\rangle }}d\lambda\right\vert
^{2}\\
&  =%
{\displaystyle\sum\limits_{\left(  \kappa,\eta\right)  \in\Lambda_{1}}}
\sum_{k_{1}\in%
\mathbb{Z}
}\left\vert \int_{\mathcal{S}}e^{2\pi i\left\langle \lambda,k_{1}\right\rangle
}F_{\kappa,\eta}\left(  \lambda\right)  d\lambda\right\vert ^{2}.
\end{align*}
Since $\mathcal{S}$ is translation congruent to $\left(  0,1\right]  $ then
$\left\{  e^{2\pi i\left\langle \lambda,k\right\rangle }\chi_{\mathcal{S}%
}\left(  \lambda\right)  :k\in%
\mathbb{Z}
\right\}  $ is an orthonormal basis for $L^{2}\left(  \mathcal{S}\right)  .$
Since $F_{\kappa,\eta}$ belongs to $L^{2}\left(  \mathcal{S}\right)  $ then
the function $k\mapsto\int_{\mathcal{S}}e^{2\pi i\left\langle \lambda
,k\right\rangle }F_{\kappa,\eta}\left(  \lambda\right)  d\lambda$ is the
Fourier transform of $F_{\kappa,\eta}$. Thus
\[%
{\displaystyle\sum\limits_{\left(  \gamma,\eta\right)  \in\Lambda}}
\left\vert \left\langle h,\tau\left(  \gamma,\eta\right)  f\right\rangle
_{L^{2}\left(  \mathbb{H}\right)  }\right\vert ^{2}=%
{\displaystyle\sum\limits_{\left(  \kappa,\eta\right)  \in\Lambda_{1}}}
\sum_{k_{1}\in%
\mathbb{Z}
}\left\vert \widehat{F_{\kappa,\eta}}\left(  k_{1}\right)  \right\vert ^{2}.
\]
Now, using Plancherel theorem on $L^{2}\left(  \mathcal{S}\right)  ,$ we
obtain%
\begin{align*}
&
{\displaystyle\sum\limits_{\left(  \gamma,\eta\right)  \in\Lambda}}
\left\vert \left\langle h,\tau\left(  \gamma,\eta\right)  f\right\rangle
_{L^{2}\left(  \mathbb{H}\right)  }\right\vert ^{2}=%
{\displaystyle\sum\limits_{\left(  \kappa,\eta\right)  \in\Lambda_{1}}}
\left\Vert \widehat{F_{\kappa,\eta}}\right\Vert _{L^{2}\left(  \mathcal{S}%
\right)  }^{2}\\
&  =%
{\displaystyle\sum\limits_{\left(  \kappa,\eta\right)  \in\Lambda_{1}}}
\int_{\mathcal{S}}\left\vert \left\langle \mathbf{P}h(\lambda),\left[
\pi_{\lambda}\left(  \kappa\right)  \otimes\overline{\pi}_{\lambda}\left(
\eta\right)  \right]  \left(  \mathbf{P}f\right)  (\lambda)\left\vert
\lambda\right\vert \right\rangle _{\mathcal{HS}}\right\vert ^{2}d\lambda.
\end{align*}
Letting $\mathbf{P}f(\lambda)=u_{\lambda}\otimes v_{\lambda},$ so that
$\left(  \mathbf{P}f(\lambda)\right)  _{\lambda\in\mathcal{S}}$ is a
measurable field of rank-one operators
\begin{align*}
&
{\displaystyle\sum\limits_{\left(  \gamma,\eta\right)  \in\Lambda}}
\left\vert \left\langle h,\tau\left(  \gamma,\eta\right)  f\right\rangle
_{L^{2}\left(  \mathbb{H}\right)  }\right\vert ^{2}\\
&  =%
{\displaystyle\sum\limits_{\left(  \kappa,\eta\right)  \in\Lambda_{1}}}
\int_{\mathcal{S}}\left\vert \left\langle \mathbf{P}h(\lambda),\left[
\pi_{\lambda}\left(  \kappa\right)  \otimes\overline{\pi}_{\lambda}\left(
\eta\right)  \right]  u_{\lambda}\otimes v_{\lambda}\left\vert \lambda
\right\vert \right\rangle _{\mathcal{HS}}\right\vert ^{2}d\lambda\\
&  =\int_{\mathcal{S}}%
{\displaystyle\sum\limits_{\left(  \kappa,\eta\right)  \in\Lambda_{1}}}
\left\vert \left\langle \mathbf{P}h(\lambda),\left[  \pi_{\lambda}\left(
\kappa\right)  \otimes\overline{\pi}_{\lambda}\left(  \eta\right)  \right]
\left(  u_{\lambda}\otimes v_{\lambda}\right)  \left\vert \lambda\right\vert
\right\rangle _{\mathcal{HS}}\right\vert ^{2}d\lambda\\
&  =\int_{\mathcal{S}}%
{\displaystyle\sum\limits_{\left(  \kappa,\eta\right)  \in\Lambda_{1}}}
\left\vert \left\langle \mathbf{P}h(\lambda),\pi_{\lambda}\left(
\kappa\right)  \left(  \left\vert \lambda\right\vert ^{1/4}u_{\lambda}\right)
\otimes\overline{\pi}_{\lambda}\left(  \eta\right)  \left(  \left\vert
\lambda\right\vert ^{1/4}v_{\lambda}\right)  \right\rangle _{\mathcal{HS}%
}\right\vert ^{2}\left\vert \lambda\right\vert d\lambda.
\end{align*}
We will now assume that $f$ is defined such that for each $\lambda
\in\mathcal{S},$ the set of systems
\begin{equation}
\left\{  \left\vert \lambda\right\vert ^{1/4}\pi_{\lambda}\left(
I_{1}\right)  u_{\lambda},\left\vert \lambda\right\vert ^{1/4}\overline{\pi
}_{\lambda}\left(  I_{2}\right)  \overline{v}_{\lambda}\right\}
\label{system}%
\end{equation}
is a set of Parseval frames where $\Lambda_{1}=I_{1}\times I_{2}$ and
\begin{align*}
I_{1} &  =\left\{  \left[
\begin{array}
[c]{cccc}%
1 & k_{3} & -k_{2} & 0\\
0 & 1 & 0 & k_{2}\\
0 & 0 & 1 & 0\\
0 & 0 & 0 & 1
\end{array}
\right]  :\left(  k_{2},k_{3}\right)  \in%
\mathbb{Z}
^{2}\right\}  ,\\
I_{2} &  =\left\{  \left[
\begin{array}
[c]{cccc}%
1 & m_{3} & -m_{2} & 0\\
0 & 1 & 0 & m_{2}\\
0 & 0 & 1 & 0\\
0 & 0 & 0 & 1
\end{array}
\right]  :\left(  m_{2},m_{3}\right)  \in%
\mathbb{Z}
^{2}\right\}  .
\end{align*}
We recall that by the density condition (see Lemma \ref{density}), and Lemma
\ref{LEM} the above fact is possible since%
\begin{align*}
\left\vert \lambda\right\vert ^{1/4}\pi_{\lambda}\left(  I_{1}\right)
u_{\lambda} &  =\mathcal{G}\left(  \left\vert \lambda\right\vert
^{1/4}u_{\lambda},%
\mathbb{Z}
\times\lambda%
\mathbb{Z}
\right)  ,\\
\text{ }\left\vert \lambda\right\vert ^{1/4}\pi_{\lambda}\left(  I_{2}\right)
\overline{v}_{\lambda} &  =\mathcal{G}\left(  \left\vert \lambda\right\vert
^{1/4}\overline{v}_{\lambda},%
\mathbb{Z}
\times\lambda%
\mathbb{Z}
\right)  ,
\end{align*}
and $\mathrm{vol}\left(
\mathbb{Z}
\times\lambda%
\mathbb{Z}
\right)  =\left\vert \lambda\right\vert \leq1.$ Next, we would like to prove
that $f$ has finite norm. Applying the Plancherel theorem, we obtain \
\[
\left\Vert f\right\Vert _{L^{2}\left(  \mathbb{H}\right)  }^{2}=\int
_{\mathcal{S}}\left\Vert \mathbf{P}f(\lambda)\right\Vert _{\mathcal{HS}}%
^{2}\left\vert \lambda\right\vert d\lambda=\int_{\mathcal{S}}\left\Vert
u_{\lambda}\right\Vert _{L^{2}\left(
\mathbb{R}
\right)  }^{2}\left\Vert v_{\lambda}\right\Vert _{L^{2}\left(
\mathbb{R}
\right)  }^{2}\left\vert \lambda\right\vert d\lambda.
\]
Since we assume that
\[
\mathcal{G}\left(  \left\vert \lambda\right\vert ^{1/4}u_{\lambda},%
\mathbb{Z}
\times\lambda%
\mathbb{Z}
\right)  ,\text{ and }\mathcal{G}\left(  \left\vert \lambda\right\vert
^{1/4}\overline{v}_{\lambda},%
\mathbb{Z}
\times\lambda%
\mathbb{Z}
\right)
\]
are Parseval frames in $L^{2}\left(
\mathbb{R}
\right)  ,$ by Lemma \ref{ONB copy(1)}
\[
\left\vert \lambda\right\vert ^{1/2}\left\Vert u_{\lambda}\right\Vert
_{L^{2}\left(
\mathbb{R}
\right)  }^{2}=\left\vert \lambda\right\vert ^{1/2}\left\Vert v_{\lambda
}\right\Vert _{L^{2}\left(
\mathbb{R}
\right)  }^{2}=\mathrm{vol}\left(
\mathbb{Z}
\times\lambda%
\mathbb{Z}
\right)  =\left\vert \lambda\right\vert .
\]
Thus,
\[
\left\Vert u_{\lambda}\right\Vert _{L^{2}\left(
\mathbb{R}
\right)  }^{2}=\left\Vert v_{\lambda}\right\Vert _{L^{2}\left(
\mathbb{R}
\right)  }^{2}=\left\vert \lambda\right\vert ^{1/2}%
\]
and
\[
\left\Vert f\right\Vert _{L^{2}\left(  \mathbb{H}\right)  }^{2}\leq\int
_{-1}^{1}\lambda^{2}d\lambda=\frac{2}{3}.
\]
Finally, applying Lemma \ref{tensor},%
\begin{align*}
&
{\displaystyle\sum\limits_{\left(  \gamma,\eta\right)  \in\Lambda}}
\left\vert \left\langle h,\tau\left(  \gamma,\eta\right)  f\right\rangle
\right\vert ^{2}\\
&  =\int_{\mathcal{S}}\overset{\overset{\left\Vert \mathbf{P}h(\lambda
)\right\Vert _{\mathcal{HS}}^{2}}{\shortparallel}}{\overbrace{%
{\displaystyle\sum\limits_{\left(  \kappa,\eta\right)  \in\Lambda_{1}}}
\left\vert \left\langle \mathbf{P}h(\lambda),\pi_{\lambda}\left(
\kappa\right)  \left(  \left\vert \lambda\right\vert ^{1/4}u_{\lambda}\right)
\otimes\overline{\pi}_{\lambda}\left(  \eta\right)  \left(  \left\vert
\lambda\right\vert ^{1/4}v_{\lambda}\right)  \right\rangle _{\mathcal{HS}%
}\right\vert ^{2}}}\left\vert \lambda\right\vert d\lambda\\
&  =\int_{\mathcal{S}}\left\Vert \mathbf{P}h(\lambda)\right\Vert
_{\mathcal{HS}}^{2}\left\vert \lambda\right\vert d\lambda=\left\Vert
h\right\Vert _{L^{2}\left(  \mathbb{H}\right)  }^{2}.
\end{align*}

\subsection{Proof of Theorem \ref{new}}

We recall that
\[
\Lambda_{1}=\left\{  \left(  \left[
\begin{array}
[c]{cccc}%
1 & k_{3} & -k_{2} & 0\\
0 & 1 & 0 & k_{2}\\
0 & 0 & 1 & 0\\
0 & 0 & 0 & 1
\end{array}
\right]  ,\left[
\begin{array}
[c]{cccc}%
1 & m_{3} & -m_{2} & 0\\
0 & 1 & 0 & m_{2}\\
0 & 0 & 1 & 0\\
0 & 0 & 0 & 1
\end{array}
\right]  \right)  :k_{i},m_{j}\in%
\mathbb{Z}
\right\}  .
\]

Let $h\in\mathbf{H}_{\mathcal{S}}.$ Let us suppose that $\tau\left(
\Lambda\right)  f$ is a Parseval frame in $\mathbf{H}_{\mathcal{S}}.$ Then
\[%
{\displaystyle\sum\limits_{\left(  \gamma,\eta\right)  \in\Lambda}}
\left\vert \left\langle h,\tau\left(  \gamma,\eta\right)  f\right\rangle
_{L^{2}\left(  \mathbb{H}\right)  }\right\vert ^{2}=\left\Vert h\right\Vert
_{L^{2}\left(  \mathbb{H}\right)  }^{2}.
\]
Now, using the fact that $\mathcal{S}$ is translation congruent to $\left(
0,1\right]  ,$ it is not too hard to check that
\begin{align*}
&
{\displaystyle\sum\limits_{\left(  \gamma,\eta\right)  \in\Lambda}}
\left\vert \left\langle h,\tau\left(  \gamma,\eta\right)  f\right\rangle
_{L^{2}\left(  \mathbb{H}\right)  }\right\vert ^{2}-\left\Vert h\right\Vert
_{L^{2}\left(  \mathbb{H}\right)  }^{2}\\
&  =%
{\displaystyle\sum\limits_{\left(  \kappa,\eta\right)  \in\Lambda_{1}}}
\int_{\mathcal{S}}\left\vert \left\langle \mathbf{P}h(\lambda),\left[
\pi_{\lambda}\left(  \kappa\right)  \otimes\overline{\pi}_{\lambda}\left(
\eta\right)  \right]  \left(  \mathbf{P}f\right)  (\lambda)\left\vert
\lambda\right\vert \right\rangle _{\mathcal{HS}}\right\vert ^{2}%
d\lambda-\left\Vert h\right\Vert _{L^{2}\left(  \mathbb{H}\right)  }^{2}\\
&  =\int_{\mathcal{S}}%
{\displaystyle\sum\limits_{\left(  \kappa,\eta\right)  \in\Lambda_{1}}}
\left\vert \left\langle \mathbf{P}h(\lambda),\left[  \pi_{\lambda}\left(
\kappa\right)  \otimes\overline{\pi}_{\lambda}\left(  \eta\right)  \right]
\left(  \mathbf{P}f\right)  (\lambda)\left\vert \lambda\right\vert
\right\rangle _{\mathcal{HS}}\right\vert ^{2}d\lambda-\int_{\mathcal{S}%
}\left\Vert \left(  \mathbf{P}h\right)  (\lambda)\right\Vert _{\mathcal{HS}%
}^{2}\left\vert \lambda\right\vert d\lambda\\
&  =\int_{\mathcal{S}}\left(
{\displaystyle\sum\limits_{\left(  \kappa,\eta\right)  \in\Lambda_{1}}}
\left\vert \left\langle \mathbf{P}h(\lambda),\left[  \pi_{\lambda}\left(
\kappa\right)  \otimes\overline{\pi}_{\lambda}\left(  \eta\right)  \right]
\left(  \mathbf{P}f\right)  (\lambda)\left\vert \lambda\right\vert
\right\rangle _{\mathcal{HS}}\right\vert ^{2}-\left\Vert \left(
\mathbf{P}h\right)  (\lambda)\right\Vert _{\mathcal{HS}}^{2}\left\vert
\lambda\right\vert \right)  d\lambda\\
&  =0.
\end{align*}
Next, replacing $h$ with $g$ such that $\mathbf{P}g(\lambda)=\chi_{B}\left(
\lambda\right)  \mathbf{P}h(\lambda)$ where $B$ is any Borel subset, then
\begin{equation}%
{\displaystyle\sum\limits_{\left(  \kappa,\eta\right)  \in\Lambda_{1}}}
\left\vert \left\langle \mathbf{P}h(\lambda),\left[  \pi_{\lambda}\left(
\kappa\right)  \otimes\overline{\pi}_{\lambda}\left(  \eta\right)  \right]
\left(  \mathbf{P}f\right)  (\lambda)\left\vert \lambda\right\vert
\right\rangle _{\mathcal{HS}}\right\vert ^{2}-\left\Vert \left(
\mathbf{P}h\right)  (\lambda)\right\Vert _{\mathcal{HS}}^{2}\left\vert
\lambda\right\vert =0\label{relation}%
\end{equation}
for every $h$ and for $d\lambda$-almost every $\lambda\in\mathcal{S}$. For
every $h$, there exists a null set $N_{h}$ such that for every $\lambda
\in\mathcal{S-}N_{h},$ (\ref{relation}) holds. However, we must show that
(\ref{relation}) holds for all $h$ and for all $\lambda$ in a conull subset
which does not depend on $h.$ In order to do so, we pick a countable dense
subset $\mathbb{A}$ of $L^{2}\left(  \mathbb{H}\right)  $ such that $\left\{
h\left(  \lambda\right)  :h\in\mathbb{A}\right\}  $ is dense in $L^{2}\left(
\mathbb{R}
\right)  \otimes L^{2}\left(
\mathbb{R}
\right)  .$ Then for
\[
\lambda\in\mathcal{S-}\left(
{\displaystyle\bigcup\limits_{h\in\mathbb{A}}}
\left(  N_{h}\right)  \right)
\]
and for $m\in\mathbb{A}$ we have that
\[%
{\displaystyle\sum\limits_{\left(  \kappa,\eta\right)  \in\Lambda_{1}}}
\left\vert \left\langle \mathbf{P}m(\lambda),\left[  \pi_{\lambda}\left(
\kappa\right)  \otimes\overline{\pi}_{\lambda}\left(  \eta\right)  \right]
\left(  \mathbf{P}f\right)  (\lambda)\left\vert \lambda\right\vert
^{1/2}\right\rangle _{\mathcal{HS}}\right\vert ^{2}=\left\Vert \left(
\mathbf{P}m\right)  (\lambda)\right\Vert _{\mathcal{HS}}^{2}.
\]
Since $\left\{  h\left(  \lambda\right)  :h\in\mathbb{A}\right\}  $ is dense
in $L^{2}\left(
\mathbb{R}
\right)  \otimes L^{2}\left(
\mathbb{R}
\right)  $ then%
\[%
{\displaystyle\sum\limits_{\left(  \kappa,\eta\right)  \in\Lambda_{1}}}
\left\vert \left\langle T,\left[  \pi_{\lambda}\left(  \kappa\right)
\otimes\overline{\pi}_{\lambda}\left(  \eta\right)  \right]  \left(
\mathbf{P}f\right)  (\lambda)\left\vert \lambda\right\vert ^{1/2}\right\rangle
_{\mathcal{HS}}\right\vert ^{2}=\left\Vert T\right\Vert _{\mathcal{HS}}^{2}%
\]
for $d\lambda$-almost every $\lambda\in\mathcal{S}$ and for every $T\in
L^{2}\left(
\mathbb{R}
\right)  \otimes L^{2}\left(
\mathbb{R}
\right)  .$

\subsection{Proof of Theorem \ref{last}}

Let $f\in\mathbf{H}_{\mathcal{S}}$ be defined such that $\mathbf{P}f=\left(
u_{\lambda}\otimes v_{\lambda}\right)  _{\lambda\in\mathcal{S}}\ $such that
$\mathcal{G}\left(  \left\vert \lambda\right\vert ^{1/4}u_{\lambda},%
\mathbb{Z}
\times\lambda%
\mathbb{Z}
\right)  $ is a Parseval Gabor frame for $d\lambda$-almost every $\lambda
\in\mathcal{S}$, and $\mathcal{G}\left(  \left\vert \lambda\right\vert
^{1/4}\overline{v}_{\lambda},%
\mathbb{Z}
\times\lambda%
\mathbb{Z}
\right)  $ is a Parseval Gabor frame for $\lambda\in\mathcal{S}$, so that (due
to Theorem \ref{main}) the system $\tau\left(  \Lambda\right)  f$ is a
Parseval frame in $\mathbf{H}_{\mathcal{S}}\mathbf{.}$ Let us define a unitary
representation of $H$ by $C:H\rightarrow\mathcal{U}\left(  L^{2}\left(
\mathbb{R}
\right)  \right)  $ such that for $\phi\in L^{2}\left(
\mathbb{R}
\right)  $%
\[
C\left(  A\right)  \phi\left(  t\right)  =\left\vert a\right\vert ^{-1/2}%
\phi\left(  a^{-1}t\right)  .
\]
Notice that $C\left(  A\right)  $ is just a unitary operator acting on
$L^{2}\left(
\mathbb{R}
\right)  $ by dilation. Let $h$ be an arbitrary element of the Hilbert space
$\mathbf{H}_{\mathcal{S}}.$ It is fairly easy to see that
\[
\mathbf{P}\left(  D_{A^{m}}h\right)  \left(  \lambda\right)  =\left\vert \det
A\right\vert ^{m/2}C\left(  A^{m}\right)  \circ\mathbf{P}h\left(  2^{m}%
\lambda\right)  \circ C\left(  A^{m}\right)  ^{-1}.
\]
To see that this holds, it suffices to perform the following computations.
Given arbitrary $u,v\in L^{2}\left(
\mathbb{R}
\right)  $ we have
\begin{align*}
\left\langle \mathbf{P}\left(  D_{A}h\right)  \left(  \lambda\right)
u,v\right\rangle  &  =\int_{\mathbb{H}}D_{A}h\left(  r\right)  \left\langle
\pi_{\lambda}\left(  r\right)  u,v\right\rangle dr\\
&  =\int_{\mathbb{H}}\left\vert \det A\right\vert ^{1/2}h\left(  r\right)
\left\langle \pi_{\lambda}\left(  ArA^{-1}\right)  u,v\right\rangle dr\\
&  =\left\vert \det A\right\vert ^{1/2}\int_{\mathbb{H}}h\left(  r\right)
\left\langle C\left(  A\right)  \pi_{2\lambda}\left(  r\right)  C\left(
A\right)  ^{-1}u,v\right\rangle dr.
\end{align*}
The last equality above is justified because for any $u\in L^{2}\left(
\mathbb{R}
\right)  ,$%
\begin{align*}
\pi_{2\lambda}\left(  x,y,z\right)  \left[  C\left(  A\right)  ^{-1}u\right]
\left(  t\right)   &  =e^{2\pi i\left(  2\lambda\right)  z}e^{-2\pi i2\lambda
yt}\left[  C\left(  A\right)  ^{-1}u\right]  \left(  t-x\right) \\
&  =\left\vert a\right\vert ^{1/2}e^{2\pi i\left(  2\lambda\right)  z}e^{-2\pi
i2\lambda yt}u\left(  a\left(  t-x\right)  \right) \\
&  =\left\vert a\right\vert ^{1/2}e^{2\pi i\lambda\left(  2z\right)  }e^{-2\pi
i\lambda ybat}u\left(  at-ax\right) \\
&  =\left\vert a\right\vert ^{1/2}\left[  \pi_{\lambda}\left(  ArA^{-1}%
\right)  u\right]  \left(  at\right) \\
&  =C\left(  A\right)  ^{-1}\left[  \pi_{\lambda}\left(  A\left(
x,y,z\right)  A^{-1}\right)  u\right]  \left(  t\right)  .
\end{align*}
So,
\[
\mathbf{P}\left(  D_{A^{m}}h\right)  \left(  \lambda\right)  =\left\vert \det
A\right\vert ^{m/2}C\left(  A^{m}\right)  \mathbf{P}h\left(  2^{m}%
\lambda\right)  C\left(  A^{m}\right)  ^{-1}.
\]
Assuming that $\mathbf{P}h\left(  \lambda\right)  =u_{\lambda}\otimes
v_{\lambda}$ is a rank-one operator, we check that
\begin{align*}
\mathbf{P}\left(  D_{A^{m}}h\right)  \left(  \lambda\right)  w  &  =\left\vert
\det A\right\vert ^{m/2}C\left(  A^{m}\right)  \left(  u_{2^{m}\lambda}\otimes
v_{2^{m}\lambda}\right)  C\left(  A^{m}\right)  ^{-1}w\\
&  =\left\vert \det A\right\vert ^{m/2}\left\langle C\left(  A^{m}\right)
^{-1}w,v_{2^{m}\lambda}\right\rangle C\left(  A^{m}\right)  u_{2^{m}\lambda}\\
&  =\left\vert \det A\right\vert ^{m/2}\left\langle w,C\left(  A^{m}\right)
v_{2^{m}\lambda}\right\rangle C\left(  A^{m}\right)  u_{2^{m}\lambda}\\
&  =\left\vert \det A\right\vert ^{m/2}\left(  C\left(  A^{m}\right)
u_{2^{m}\lambda}\otimes C\left(  A^{m}\right)  v_{2^{m}\lambda}\right)  w.
\end{align*}
As a result,
\[
\mathbf{P}\left(  D_{A^{m}}h\right)  \left(  \lambda\right)  =\left\vert \det
A\right\vert ^{m/2}\left(  C\left(  A\right)  u_{2^{m}\lambda}\otimes C\left(
A\right)  v_{2^{m}\lambda}\right)  .
\]
Thus, $D_{A^{m}}\left(  \mathbf{H}_{\mathcal{S}}\right)  =\mathbf{H}%
_{2^{-m}\left(  \mathcal{S}\right)  }$ and for $m\neq k,$ $\mathbf{H}%
_{2^{-m}\left(  \mathcal{S}\right)  }$ is orthogonal to $\mathbf{H}%
_{2^{-k}\left(  \mathcal{S}\right)  }.$ Also, since the system $\tau\left(
\Lambda\right)  f$ is a Parseval frame (see Theorem \ref{main}) for
$\mathbf{H}_{\mathcal{S}}$ and since $D_{A^{m}}$ is unitary, then
\[
D_{A^{m}}\left(  \tau\left(  \Lambda\right)  f\right)  =\tau\left(
A^{m}\left(  \Lambda\right)  \right)  D_{A^{m}}f
\]
is a Parseval frame in the Hilbert space $\mathbf{H}_{2^{-m}\left(
\mathcal{S}\right)  }.$ Finally, since $\mathcal{S}$ is dilation congruent to
the set $\left[  -1,-1/2\right)  \cup\left(  1/2,1\right]  $ then $%
{\displaystyle\bigcup\limits_{m\in\mathbb{Z}}}
\mathbf{H}_{2^{-m}\left(  \mathcal{S}\right)  }$ is dense in $L^{2}\left(
\mathbb{H}\right)  $ and
\[
\overline{%
{\displaystyle\bigcup\limits_{m\in\mathbb{Z}}}
\mathbf{H}_{2^{-m}\left(  \mathcal{S}\right)  }}=L^{2}\left(  \mathbb{H}%
\right)  .
\]
We conclude that the given system
\[
\left\{  D_{A^{m}}\tau\left(  \gamma,\eta\right)  f:m\in%
\mathbb{Z}
,\left(  \gamma,\eta\right)  \in\Lambda\right\}
\]
is a Parseval frame for $L^{2}\left(  \mathbb{H}\right)  $.

\begin{acknowledgement}
I thank Hartmut F\"{u}hr for interesting conversations during the CMS meeting
in Halifax, NS Canada.
\end{acknowledgement}

\end{document}